\documentclass[11pt]{amsart}
\usepackage{amsmath, amsthm, amssymb,amscd}
\usepackage{float}
\usepackage{graphicx}
\usepackage{xypic}
\usepackage[arrow, matrix, curve]{xy}
\usepackage{color}
\usepackage{enumitem}  
\usepackage{tikz-cd} 
\usepackage{url}

\usepackage{blindtext}

\usepackage[english]{babel}
\newcommand{\seq}{\subseteq}
\newcommand{\C}{\mathbb{C}}

\newtheorem{thm}{Theorem}[section]
\newtheorem*{thm-nl}{Theorem}
\newtheorem*{prop-nl}{Proposition}

\newtheorem{lem}[thm]{Lemma}

\def\PP{{\mathbf P}}

\def\Pic0{{\rm Pic}^0(X)}

\newtheorem*{cor-nl}{Corollary}

\newtheorem*{conjecture-nl}{Conjecture}

\newtheorem*{quest-nl}{Question}
\newtheorem*{quests-nl}{Questions}

\newtheorem{prop}[thm]{Proposition}

\theoremstyle{remark}
\newtheorem*{rem}{Remark}

\title{{A Simple Proof of Voisin's Theorem for Canonical Curves of Even Genus}}
\date{\today}

\author[M. Kemeny]{Michael Kemeny}

\address{University of Wisconsin-Madison, Department of Mathematics, 480 Lincoln Dr
\hfill \newline\texttt{}
 \indent WI 53706, USA} \email{{\tt michael.kemeny@gmail.com}}

\begin{document}
\begin{abstract}
We give a simple proof of Voisin's Theorem for general canonical curves of even genus. This completely determines the terms of the minimal free resolution of the coordinate ring of such curves.
\end{abstract}
\maketitle
\setcounter{section}{-1}
\section{Introduction}
The aim of this paper is to give a simple proof of a theorem of Voisin \cite{V1} on the equations of general canonical curves of even genus. Recall that the classical Theorem of Noether--Babbage--Petri states that canonical curves are projectively normal, and that the ideal $I_{C/\PP^{g-1}}$ is generated by quadrics (with a few, well-understood, exceptions), see \cite{arbarello-sernesi-petri} for a modern treatment. In the 1980s, M.\ Green realized that these classical results about the equations defining canonical curves should be the first case of a much more general statement about higher syzygies, and he made a very influential conjecture \cite{green-koszul} in this direction.\medskip

Whilst the general case of Green's Conjecture remains open, in 2002 Voisin made a breakthrough by proving the conjecture for \emph{general} curves of even genus \cite{V1}.  Voisin's argument relies on an intricate study of the geometry of Hilbert schemes on a K3 surface. Recently, an alternative proof of Voisin's Theorem was given in \cite{AFPRW}. In this proof, the authors first degenerate the K3 surface to the tangent developable. The proof then relies on a sophisticated use of the representation theory of the $SL_2$ action present in this special situation. In particular, various explicit plethysm formulae play a key role. \medskip

In this paper, we give a simple proof of Voisin's Theorem, using only basic homological algebra. We achieve this by a direct computation using K3 surfaces. Let $X$ be a K3 surface over $\C$ with Picard group generated by an ample line bundle $L$ of even genus $g=2k$, i.e.\ $(L)^2=2g-2$. Define $K_{p,q}(X,L)$ as the middle cohomology of 
$$\bigwedge^{p+1} \rm{H}^0(X,L) \otimes \rm{H}^0(X,L^{\otimes q-1}) \to \bigwedge^p \rm{H}^0(X,L) \otimes \rm{H}^0(X,L^{\otimes q}) \to \bigwedge^{p-1} \rm{H}^0(X,L) \otimes \rm{H}^0(X,L^{\otimes q+1}) $$\\
 Voisin's Theorem states that $K_{k,1}(X,L)=0$, \cite{V1}. This single vanishing suffices to prove Green's Conjecture for general canonical curves in even genus. \medskip

%Precisely, set $R:=\rm{Sym}\; H^0(\omega_C)$ and take the resolution $0 \leftarrow \Gamma_C(\omega_C) \leftarrow F_{\bullet}$ of a general curve of genus $g=2k \geq 4$. Writing $F_p=\bigoplus_{q \geq 0} R(-p-q)^{\oplus b_{p,q}}$, the nonzero Betti numbers are $b_{0,0}=b_{2k-2,3}=1$ and
%\begin{align*}
%b_{p,1}&=\frac{(2k-p-1)(2k-2p-3)}{p+1} \binom{2k-1}{p-1}, \; \; \text{for $p \leq k-1$},\\
%b_{p,2}&=\frac{(2k-p-2)(2p-2k+3)}{p+2} \binom{2k-1}{p}, \; \; \text{for $k-1 \leq p \leq 2k-3$}
%\end{align*}\\
%The idea of using K3 surfaces to prove Generic Green's Conjecture goes back at least to \cite{green-lazarsfeld-special}. 

Our proof is short and direct. Let $E$ be the rank two \emph{Lazarsfeld--Mukai bundle} associated to a $g^1_{k+1}$ on a smooth curve $C \in |L|$, see \cite{lazarsfeld-BNP}. The dual bundle $E^{\vee}$ fits into the exact sequence
$$0 \to E^{\vee} \to \rm{H}^0(C,A) \otimes \mathcal{O}_X \to i_*A \to 0,$$
for $A$ a $g^1_{k+1}$ on $C$, where $i: C \hookrightarrow X$ is the inclusion. The vector bundle $E$ has invariants $\det(E)=L$, $h^0(E)=k+2$, $h^1(E)=h^2(E)=0$.\medskip

We deduce Voisin's Theorem from the K\"unneth formula on $X \times \PP(\rm{H}^0(E))$. Our proof quickly reduces to showing that a certain square matrix is nonsingular. However, our matrix takes the form $\rm{H}^i(\rm{Sym}^j \mathcal{F}) \to \rm{H}^i(\rm{Sym}^{j-1} \mathcal{F} \otimes \mathcal{F})$, for some vector bundle $\mathcal{F}$, so that the desired nonsingularity is \emph{automatic}, see Proposition \ref{auto-injectivity}.\medskip
%By realizing all linear maps geometrically as natural maps between cohomology of vector bundles, we avoid many technical difficulties. For instance, we do not have any difficult commutativities to check. %%%%I previously put this in, but decided it was unnecessary%%

The starting point for our new approach is based on a variation of the secant construction from \cite[\S 3]{ein-lazarsfeld-asymptotic}. With the same method (but using a different exterior power), we proved in \cite{geometric-syzygy} the Geometric Syzygy Conjecture, a higher syzygy version of Green's Theorem on Quadrics, \cite{green-quadrics}, in even genus. We further expect that our approach will generalize well to the study of syzygies of higher dimensional varieties. \medskip

Green's Conjecture applied to general curves is particularly strong in that it completely determines the terms in the resolution of the canonical ring $\Gamma_C(\omega_C):=\bigoplus_q \rm{H}^0(\omega_C^{\otimes q})$, see \cite[\S 4.1]{farkas-progress}. Note that, for a general curve, Green's Conjecture naturally breaks down into even and odd genus cases.  A few years after her breakthrough for even genus curves, Voisin deduced the odd genus case of Green's conjecture out of the even genus case \cite{V2}. We hope to port our method to this setting, but have not yet achieved this. It is very likely that our proof works in characteristic $p$ for large $p$, but we have not checked this.\medskip

 The reader may notice a resemblance between our argument and \cite{V2}, Proof of Proposition 8. The idea of considering $\text{Sym}^{k+1} \mathcal{S}$ was foreshadowed in \cite{AFPRW}. Other contributions to this topic include \cite{aprodu-farkas}, which relies on Voisin's results, and \cite{raicu-sam}, which relies instead upon the results of \cite{AFPRW}.

\subsection{Preliminaries}
We gather here a few facts. Let $0 \to F_1 \to F_2 \to F_3 \to 0$ be a short exact sequence of vector bundles over $\C$. From \cite{weyman-sym-ext}, for any $i$ we have exact sequences
\begin{align*}
&\ldots \to \bigwedge^{i-2} F_2 \otimes \rm{Sym}^2(F_1) \to \bigwedge^{i-1} F_2 \otimes F_1  \to \bigwedge^i F_2 \to \bigwedge^i F_3 \to 0, \\
&0 \to \rm{Sym}^i(F_1) \to \rm{Sym}^{i}(F_2) \to \rm{Sym}^{i-1}(F_2) \otimes F_3 \to \rm{Sym}^{i-2}(F_2) \otimes \bigwedge^2 F_3 \to \ldots
\end{align*}

We state two formulae which we will often use without specific mention. Let $f: X \to Y$ be a morphism of varieties and $\mathcal{F} \in \rm{Coh}(X)$ a sheaf. If $\mathcal{E}$ is a vector bundle on $Y$ then we have the \emph{Projection Formula} $\rm{R}^if_*(\mathcal{F} \otimes f^* \mathcal{E}) \simeq \rm{R}^if_*\mathcal{F} \otimes \mathcal{E}$, \cite[III, Ex.\ 8.3]{hartshorne}. In particular, $f_*(\mathcal{F} \otimes f^* \mathcal{E}) \simeq f_*\mathcal{F} \otimes \mathcal{E}$. If $\rm{R}^if_*\mathcal{F}=0$ for all $ i>0$ then $\rm{H}^p(X,\mathcal{F}) \simeq \rm{H}^p(Y, f_* \mathcal{F})$ for $p \geq 0$, \cite[III, Ex.\ 8.1 ]{hartshorne}.\medskip

 If $X, Y$ are varieties and $\mathcal{F} \in \rm{Coh}(X), \mathcal{G} \in \rm{Coh}(Y)$ are sheaves, the \emph{K\"unneth formula} states
 $$\rm{H}^{m}(X \times Y, \mathcal{F} \boxtimes \mathcal{G}) \simeq \bigoplus_{a+b=m} \rm{H}^a(X,\mathcal{F}) \otimes \rm{H}^b(Y, \mathcal{G}),$$
where $\mathcal{F} \boxtimes \mathcal{G}:=p^*\mathcal{F} \otimes q^* \mathcal{G}$, for projections $p: X \times Y \to X$, $q: X \times Y \to Y$.\medskip
 
 Assume we have an exact sequence $0 \to \mathcal{F} \to \mathcal{G} \to \mathcal{H} \to 0$ of coherent sheaves on a smooth variety, with $\mathcal{G}$ locally free. Assume either $\mathcal{H}$ is locally free or $\mathcal{H} \simeq \mathcal{O}_D$ for a divisor $D$. Then $\mathcal{F}$ is locally free. This follows from \cite[III, Ex 6.5]{hartshorne}. 

\section{The proof}
Consider the unique rank two, Lazarsfeld--Mukai, bundle $E$ on $X$ as in the introduction. For general $s \in \rm{H}^0(E)$, the zero-locus $Z(s)$ corresponds to a $g^1_{k+1}$ on a smooth $C \in |L|$. For \emph{any} $s \in \rm{H}^0(E)$, $Z(s) \seq X$ is zero-dimensional and we have an exact sequence $$0 \to \mathcal{O}_X \xrightarrow{s} E \xrightarrow{\wedge s} I_{Z(s)} \otimes L \to 0.$$
Set $\PP:=\PP(\rm{H}^0(E))\simeq \PP^{k+1}$. Consider $X \times \PP$ with projections $p: X \times \PP \to X$, $q: X \times \PP \to \PP$. Define $\mathcal{Z} \seq X \times \PP$ as the locus $\left\{ (x,s) \; | \; s(x)=0 \right\}$. Since $E$ is globally generated, $\mathcal{Z}$ is a projective bundle over $X$ and hence smooth. We have an exact sequence
$$0 \to \mathcal{O}_X \boxtimes \mathcal{O}_{\PP}(-2) \xrightarrow{\rm{id}} E \boxtimes \mathcal{O}_{\PP}(-1) \to p^*L \otimes I_{\mathcal{Z}} \to 0,$$
where the first nonzero map is given by multiplication by $$\rm{id} \in \rm{H}^0\left(E \boxtimes \mathcal{O}_{\PP}(1)\right) \simeq \rm{H}^0(E) \otimes \rm{H}^0(E)^{\vee} \simeq \rm{Hom}\left(\rm{H}^0(E), \rm{H}^0(E)\right).$$ 
Note that $\mathcal{Z} \to \PP$ is finite and flat, \cite[Prop.\ 6.1.5]{EGA}.

\begin{rem}
As soon as there exists a nontrivial, effective divisor $C$ on $X$ with $\rm{H}^0(E(-C))$ nonzero, then $\mathcal{Z} \to \PP$ cannot be finite and flat. For this reason, it is essential that $\rm{Pic}(X) \simeq \mathbb{Z}[L]$.
\end{rem}
\medskip

Let $\mathcal{M}:=p^* M_L$, where $M_L$ is the Kernel Bundle $0 \to M_L \to \rm{H}^0(L)\otimes \mathcal{O}_X \to L \to 0.$
By the well-known Kernel Bundle description of Koszul cohomology \cite[\S 3]{ein-lazarsfeld-asymptotic}, it suffices to show $$\rm{H}^1(X, \bigwedge^{k+1} M_L )=0.$$
Note $\rm{H}^1(X \times \PP,\bigwedge^{k+1}\mathcal{M}) \simeq \rm{H}^1(X, \bigwedge^{k+1}M_L) $ by the K\"unneth formula, as $\rm{H}^1(\mathcal{O}_X)=0$.\\

We now adapt \cite[p.\ 615]{ein-lazarsfeld-asymptotic}. Let $\pi:B \to  X \times \PP $ be the blow-up along $\mathcal{Z}$ with exceptional divisor $D$. Then 
$\pi_*\mathcal{O}_B \simeq \mathcal{O}_{X \times \PP}, \; \; \pi_* I_{D} \simeq I_{\mathcal{Z}} \; \; \text{and} \; \; \rm{R}^i \pi_*\mathcal{O}_B=\rm{R}^i\pi_* I_D=0 \; \text{for $i>0$},$
cf.\ \cite[V, Prop.\ 3.4 and Ex.\ 3.1]{hartshorne}. Set $p':=p \circ \pi$, $q':=q \circ \pi$. We have canonical identifications
$$q'_*({p'}^*L \otimes I_D) \simeq q_*(p^*L \otimes I_{\mathcal{Z}}), \; \; \; \; \; \; \; q'_*{p'}^*L  \simeq q_* p^*L.$$
\smallskip

Consider 
$$\mathcal{W}:=\text{Coker}\left(q'_*({p'}^*L \otimes I_D) \to q'_*{p'}^*L\right)\simeq  \text{Coker}\left(q_*({p}^*L \otimes I_{\mathcal{Z}}) \to q_*{p}^*L\right).$$
\begin{lem}
The sheaf $\mathcal{W}$ is locally free of rank $k$.
\end{lem}
\begin{proof}
Applying $\rm{R}q_*$ we have the exact sequence $0 \to \mathcal{W} \to q_*(L_{|_{\mathcal{Z}}}) \to \rm{R}^1q_*(p^*L \otimes I_{\mathcal{Z}})\to 0$. For any $s \in \rm{H}^0(E)$, we have $\rm{H}^1(X,L\otimes I_{Z(s)}) \simeq \rm{H}^2(\mathcal{O}_X)\simeq \C$. Thus $q_*(L_{|_{\mathcal{Z}}}) $ and $\rm{R}^1q_*(p^*L \otimes I_{\mathcal{Z}})$ are locally free of ranks $k+1$ and $1$, respectively, by Grauert's Theorem \cite[III, \S 12]{hartshorne}. The claim follows.
\end{proof}
\smallskip

Since $D$ is a divisor, we have a rank $k$ vector bundle 
$\displaystyle{\Gamma:=\text{Ker}\left({q'}^*\mathcal{W} \twoheadrightarrow {p'}^*L_{|_D} \right).}$
Let $\mathcal{S}$ be the vector bundle on $B$ defined by the exact sequence
$$0 \to \mathcal{S} \to {q'}^* q'_* ({p'}^*L \otimes I_D) \to {p'}^*L \otimes I_D \to 0.$$
We have an exact sequence $0 \to \mathcal{S} \to \pi^* \mathcal{M} \to \Gamma \to 0,$ giving the exact sequence
$$ \ldots \to \bigwedge^{k-1} \pi^* \mathcal{M} \otimes \rm{Sym}^2 \mathcal{S} \to  \bigwedge^{k} \pi^* \mathcal{M} \otimes \mathcal{S} \to \bigwedge^{k+1}\pi^* \mathcal{M} \to 0.$$
\begin{rem}
The Secant Sheaves defined in \cite[\S 3]{ein-lazarsfeld-asymptotic} are only torsion-free in general. To apply \cite{weyman-sym-ext}, we need $\Gamma$ to be locally free and hence we must pass to the blow-up $B$.
\end{rem}
 
 To prove Voisin's Theorem it suffices to show $$\rm{H}^i(B, \bigwedge^{k+1-i} \pi^* \mathcal{M} \otimes \rm{Sym}^i \mathcal{S})=0 \; \; \text{for $1 \leq i \leq k+1$}.$$
One readily shows these vanishings for $i <k+1$ (see Theorem \ref{main-thm}). The crucial point is to show $\rm{H}^{k+1}(B, \rm{Sym}^{k+1} \mathcal{S})=0$. To ease the notation, we set $$\mathcal{G}:=q^*q_*(p^*L \otimes I_{\mathcal{Z}}).$$
\begin{lem} \label{first-iso}
We have natural isomorphisms
\begin{align*}
\rm{H}^k(X \times \PP,\rm{Sym}^{k+1} \mathcal{G})&\simeq \rm{Sym}^k \rm{H}^0(E)\\
\rm{H}^{k+1}(X \times \PP,\rm{Sym}^{k+1} \mathcal{G})&=0
\end{align*}
\end{lem}
\begin{proof}
The exact sequence $0 \to q^*\mathcal{O}(-2) \to q^*q_*p^*E \otimes q^*\mathcal{O}(-1) \to \mathcal{G}\to 0$
 gives the exact sequence
$$0 \to \rm{Sym}^k q^*q_*p^*E \otimes q^*\mathcal{O}(-k-2) \to \rm{Sym}^{k+1}q^*q_*p^*E \otimes q^*\mathcal{O}(-k-1) \to \rm{Sym}^{k+1}\mathcal{G} \to 0.$$
Since $q_*p^*E \simeq \rm{H}^0(E) \otimes \mathcal{O}_{\PP}$ is trivial,
\begin{align*}
\rm{H}^k(\rm{Sym}^{k+1}\mathcal{G}) \simeq \rm{H}^{k+1}(\rm{Sym}^k q^*q_*p^*E \otimes q^*\mathcal{O}(-k-2))\simeq \rm{Sym}^k \rm{H}^0(E)
\end{align*}
The vanishing $\rm{H}^{k+1}(\rm{Sym}^{k+1} \mathcal{G})=0$ follows from $$\rm{H}^{k+1}(\mathcal{O}_X \boxtimes \mathcal{O}_{\PP}(-k-1))=\rm{H}^{k+2}(\mathcal{O}_X \boxtimes \mathcal{O}_{\PP}(-k-2))=0,$$
using $\rm{H}^1(\mathcal{O}_X)=0$.
\end{proof}

The next lemma is a similar computation to the previous one.
\begin{lem} \label{first-twist-lem}
We have a natural isomorphism
$\rm{H}^k(\rm{Sym}^{k} \mathcal{G} \otimes p^*L \otimes I_{\mathcal{Z}}) \simeq \rm{Sym}^k \rm{H}^0(E).$

\end{lem}
\begin{proof}
We have the short exact sequence
\begin{align*}
0 &\to \rm{Sym}^{k-1} q^*q_*p^*E \otimes q^*\mathcal{O}(-k-1) \otimes p^*L \otimes I_{\mathcal{Z}} \to \rm{Sym}^{k}q^*q_*p^*E \otimes q^*\mathcal{O}(-k) \otimes p^*L \otimes I_{\mathcal{Z}} \\
&\to \rm{Sym}^{k}\mathcal{G} \otimes p^*L \otimes I_{\mathcal{Z}}\to 0,
\end{align*}
as well as the exact sequence
$0 \to \mathcal{O}_X \boxtimes \mathcal{O}_{\PP}(-2) \to E \boxtimes \mathcal{O}_{\PP}(-1) \to p^*L \otimes I_{\mathcal{Z}} \to 0.$\\
%The first nonzero map in this sequence is given by multiplication by the section $$\rm{id} \in H^0(E \boxtimes \mathcal{O}_{\PP}(1)) \simeq H^0(E) \otimes H^0(E)^{\vee} \simeq \rm{Hom}(H^0(E), H^0(E)).$$ \\

%We claim that $H^{k+1}((L \boxtimes \mathcal{O}_{\PP}(-k-1))\otimes I_{\mathcal{Z}})=0$.

By the K\"unneth formula, $\rm{H}^{k+2}(\mathcal{O}_X \boxtimes \mathcal{O}_{\PP}(-k-3))=0$. We have
$$\rm{H}^{k+1}(\mathcal{O}_X \boxtimes \mathcal{O}_{\PP}(-k-3))=\rm{H}^{k+1}(K_{\PP}(-1)) \simeq \rm{H}^0(\mathcal{O}_{\PP}(1))^{\vee}\simeq \rm{H}^0(E).$$
Further, $\rm{H}^{k+1}(E \boxtimes \mathcal{O}_{\PP}(k-2)) \simeq \rm{H}^{k+1}(E \boxtimes K_{\PP}) \simeq \rm{H}^0(E).$
The map $$\rm{H}^{k+1}(\mathcal{O}_X \boxtimes \mathcal{O}_{\PP}(-k-3)) \to \rm{H}^{k+1}(E \boxtimes \mathcal{O}_{\PP}(-k-2)) $$ is identified with $\rm{id}: \rm{H}^0(E) \to \rm{H}^0(E)$. Thus $\rm{H}^{k+1}(L \boxtimes \mathcal{O}_{\PP}(-k-1)\otimes I_{\mathcal{Z}})=0.$ We likewise have $\rm{H}^{k}(L \boxtimes \mathcal{O}_{\PP}(-k-1)\otimes I_{\mathcal{Z}})=0.$ \medskip

%For this, combined with the above, it suffices to show $H^k(E \boxtimes \mathcal{O}_{\PP}(-k-2))=0$, which follows from the K\"unneth formula.\\

Using that $q_*p^*E \simeq \rm{H}^0(E) \otimes \mathcal{O}_{\PP}$ is trivial, we have
\begin{align*}
\rm{H}^k(\rm{Sym}^{k} \mathcal{G} \otimes p^*L \otimes I_{\mathcal{Z}}) &\simeq \rm{H}^k(\rm{Sym}^k q^*q_*p^*E \otimes q^*\mathcal{O}(-k)\otimes p^*L \otimes I_{\mathcal{Z}}) \\
&\simeq \rm{Sym}^k\rm{H}^0(E)\otimes \rm{H}^k(L \boxtimes \mathcal{O}_{\PP}(-k)\otimes I_{\mathcal{Z}}).
\end{align*}
To finish the proof, it suffices to show that the boundary map
$$\rm{H}^k(L \boxtimes \mathcal{O}_{\PP}(-k)\otimes I_{\mathcal{Z}}) \to \rm{H}^{k+1}(q^*K_{\PP})$$
is an isomorphism, which follows from the fact that $\rm{H}^i(E \boxtimes \mathcal{O}(-k-1))=0$ for all $i$.
\end{proof}

We now repeat the previous lemma, twisting instead by $\mathcal{G}:=q^*q_*(p^*L \otimes I_{\mathcal{Z}})$.
\begin{lem} \label{second-twist-lem}
The evaluation morphism $\mathcal{G} \twoheadrightarrow p^*L \otimes I_{\mathcal{Z}}$ induces an isomorphism
$$\rm{H}^k(\rm{Sym}^{k} \mathcal{G} \otimes \mathcal{G}) \xrightarrow{\sim} \rm{H}^k(\rm{Sym}^{k} \mathcal{G} \otimes p^*L \otimes I_{\mathcal{Z}}) .$$
\end{lem}
\begin{proof}
We have the short exact sequences
\begin{align*}
0 \to \rm{Sym}^{k-1} q^*q_*p^*E \otimes q^*\mathcal{O}(-k-1) \otimes \mathcal{G} \to \rm{Sym}^{k}q^*q_*p^*E \otimes q^*\mathcal{O}(-k) \otimes \mathcal{G} \to \rm{Sym}^{k}\mathcal{G} \otimes \mathcal{G} \to 0,
\end{align*}
and $0 \to q^*\mathcal{O}(-2) \to q^*q_*p^*E \otimes q^*\mathcal{O}(-1) \to \mathcal{G} \to 0.$\medskip

Using the second sequence, $\rm{H}^{k+1}(\mathcal{G} \otimes \mathcal{O}_{\PP}(-k-1))=\rm{H}^{k}(\mathcal{G} \otimes \mathcal{O}_{\PP}(-k-1))=0$ and $\rm{H}^k(\mathcal{G} \otimes q^*\mathcal{O}(-k))\xrightarrow{\sim} \rm{H}^{k+1}(q^*K_{\PP})$. Thus
\begin{align*}
\rm{H}^k(\rm{Sym}^{k} \mathcal{G} \otimes \mathcal{G}) \simeq \rm{H}^k(\rm{Sym}^k q^*q_*p^*E \otimes q^*\mathcal{O}(-k)\otimes \mathcal{G}) \simeq \rm{Sym}^k\rm{H}^0(E)
\end{align*}
and the evaluation map gives an isomorphism $\rm{H}^k(\rm{Sym}^{k}\mathcal{G} \otimes \mathcal{G}) \xrightarrow{\sim} \rm{H}^k(\rm{Sym}^{k} \mathcal{G} \otimes p^*L \otimes I_{\mathcal{Z}}).$

\end{proof}
As a corollary, we now deduce:
\begin{prop} \label{auto-injectivity}
The natural map gives an isomorphism $$\rm{H}^k(\rm{Sym}^{k+1} \mathcal{G}) \xrightarrow{\sim} \rm{H}^k(\rm{Sym}^{k} \mathcal{G} \otimes p^*L \otimes I_{\mathcal{Z}}).$$
\end{prop}
\begin{proof}
By the previous lemmas, it suffices to show that the natural morphism
$$\rm{H}^k(\rm{Sym}^{k+1} \mathcal{G}) \to \rm{H}^k(\rm{Sym}^{k} \mathcal{G} \otimes \mathcal{G})$$
is injective. For a vector bundle $\mathcal{F}$ the composition $\rm{Sym}^i \mathcal{F} \to \rm{Sym}^{i-1}\mathcal{F} \otimes \mathcal{F} \to \rm{Sym}^i \mathcal{F}$
of natural maps is just multiplication by $i$. This completes the proof.
\end{proof}
We now complete the proof that $K_{k,1}(X,L)=0$.
\begin{thm} \label{main-thm}
We have $\rm{H}^i(B, \bigwedge^{k+1-i} \pi^* \mathcal{M} \otimes \rm{Sym}^i \mathcal{S})=0$ for $1 \leq i \leq k+1$.
\end{thm}
\begin{proof}
Observe $\pi^*\mathcal{G}\simeq {q'}^*q'_*({p'}^*L \otimes I_{D})$. From the defining sequence for $\mathcal{S}$, we have an exact sequence
\begin{align} \label{ses-S}
0 \to \rm{Sym}^{i}\mathcal{S} \to \rm{Sym}^{i} \pi^* \mathcal{G} \to \rm{Sym}^{i-1} \pi^*\mathcal{G} \otimes {p'}^*L \otimes I_{D} \to 0.
\end{align}
Using the projection formula, and recalling the identities $\pi_*\mathcal{O}_B \simeq \mathcal{O}_{X \times \PP}$, $\pi_* I_{D} \simeq I_{\mathcal{Z}}$ and $\rm{R}^j \pi_*\mathcal{O}_B=\rm{R}^j\pi_* I_D=0$ for $j>0$,  we may identify
$$\rm{H}^{\ell}(B, \rm{Sym}^{i} \pi^* \mathcal{G}) \to \rm{H}^{\ell}(B,\rm{Sym}^{i-1} \pi^*\mathcal{G} \otimes {p'}^*L \otimes I_{D})$$
with the natural map $\rm{H}^{\ell}(X \times \PP, \rm{Sym}^{i} \mathcal{G}) \to \rm{H}^{\ell}(X \times \PP,\rm{Sym}^{i-1} \mathcal{G} \otimes {p}^*L \otimes I_{\mathcal{Z}})$, for any $\ell$. Taking the long exact sequence of cohomology for the sequence (\ref{ses-S}) for $i=k+1$ and applying the previous lemmas, we immediately see $\displaystyle \rm{H}^{k+1}(B,\rm{Sym}^{k+1} \mathcal{S})=0.$ \\

To complete the proof, it suffices to show
\begin{align*}
\rm{H}^i (\bigwedge^{k+1-i} \pi^* \mathcal{M} \otimes \rm{Sym}^{i} \pi^* \mathcal{G})=\rm{H}^{i-1}(\bigwedge^{k+1-i} \pi^* \mathcal{M} \otimes \rm{Sym}^{i-1} \pi^* \mathcal{G} \otimes {p'}^*L \otimes I_{D} )=0, \; \; \text{for $1 \leq i \leq k$}.
\end{align*}
The first vanishing follows from the exact sequence 
$$0 \to \rm{Sym}^{i-1} q^*q_*p^*E \otimes q^*\mathcal{O}(-i-1) \to \rm{Sym}^{i}q^*q_*p^*E \otimes q^*\mathcal{O}(-i) \to \rm{Sym}^{i}\mathcal{G}\to 0,$$
together with $\rm{H}^i(\bigwedge^{k+1-i} M_L \boxtimes \mathcal{O}_{\PP}(-i))=\rm{H}^{i+1}(\bigwedge^{k+1-i}M_L \boxtimes \mathcal{O}_{\PP}(-i-1))=0$ for $1 \leq i \leq k$.\smallskip

Next, from the exact sequence
\begin{align*}
0 &\to \rm{Sym}^{i-2} q^*q_*p^*E \otimes q^*\mathcal{O}(-i) \otimes p^*L \otimes I_{\mathcal{Z}} \to \rm{Sym}^{i-1}q^*q_*p^*E \otimes q^*\mathcal{O}(-i+1) \otimes p^*L \otimes I_{\mathcal{Z}} \\
&\to \rm{Sym}^{i-1}\mathcal{G} \otimes p^*L \otimes I_{\mathcal{Z}}\to 0,
\end{align*}
it suffices to show
$$\rm{H}^{i-1}(\bigwedge^{k+1-i} M_L (L)\boxtimes \mathcal{O}(-i+1)\otimes I_{\mathcal{Z}})=\rm{H}^{i}(\bigwedge^{k+1-i} M_L(L) \boxtimes \mathcal{O}(-i) \otimes I_{\mathcal{Z}})=0, \; \; 1 \leq i \leq k.$$ This follows from $0 \to \mathcal{O}_X \boxtimes \mathcal{O}_{\PP}(-2) \to E \boxtimes \mathcal{O}_{\PP}(-1) \to L \otimes I_{\mathcal{Z}} \to 0,$ as $\rm{H}^0(X,M_L)=0$ if $i=k$.

\end{proof}
\textbf{Acknowledgements}  I thank G.\ Farkas for numerous discussions on these topics. I thank R.\ Lazarsfeld for encouragement and for detailed comments on a draft. The author is supported by NSF grant DMS-1701245.


\begin{thebibliography}{aaaaaa}
%\bibitem[A]{accola} R. D. M. Accola, {\em{On Castelnuovo's inequality for algebraic curves. I}}, Transactions of the American Mathematical Society \textbf{251} (1979), 357-373.
%\bibitem[ACV]{twisted} D. Abramovich, A. Corti, and A. Vistoli, { \em{Twisted bundles and admissible covers}}, Communications in Algebra \textbf{31} (2003), 3547-3618.
%\bibitem[Ap1]{aprodu-higher} M. Aprodu, {\em{On the vanishing of higher syzygies of curves}}, Mathematische Zeitschrift \textbf{241} (2002), 1-15.
%\bibitem[Ap2]{aprodu-remarks} M. Aprodu, {\em{Remarks on syzygies of $d$-gonal curves}}, Math. Res. Lett \textbf{12} (2005), 387-400.
%\bibitem[AF]{aprodu-farkas} M. Aprodu and G. Farkas, {\em{Green's Conjecture for smooth curves on arbitrary K3 surfaces}}, Compositio Math.\ \textbf{147} (2011), 839-851.
%\bibitem[AK]{ara-kol} C. Araujo and J.  Koll\'{a}r, {\em{Rational Curves on Varieties}}, Higher Dimensional Varieties and Rational Points. Springer (2003).
%\bibitem[AH]{arbarello-harris} E. Arbarello and J. Harris, {\em{Canonical curves and quadrics of rank 4}}, Compositio Math. \textbf{43} (1981), 145-179.
%\bibitem[AM]{andreotti-mayer} A. Andreotti and A. L. Mayer, {\em{On period relations for abelian integrals on algebraic curves}}, Ann.\ Sc.\ Norm.\ Sup.\ Pisa (1967), 189-238.

%\bibitem[AN1]{aprodu-nagel-nonvanishing} M. Aprodu and  J. Nagel, {\em{Non-vanishing for Koszul cohomology of curves}}, Commentarii Math. Helv. \textbf{82} (2007), 617-628

%\bibitem[AN2]{aprodu-nagel} M. Aprodu and  J. Nagel, {\em{Koszul cohomology and algebraic geometry}}, University Lecture Series \textbf{52}, American Mathematical Society, Providence, RI (2010).

%\bibitem[AM]{andreotti-mayer} A. Andreotti and A. L. Mayer, {\em{On period relations for abelian integrals on algebraic curves}}, Annali della Scuola Normale Superiore di Pisa \textbf{21} (1967),189-238.
%\bibitem[AS]{aprodu-sernesi} M. Aprodu and  E. Sernesi, {\em{Secant spaces and syzygies of special line bundles on curves}}, Algebra and Number Theory \textbf{9} (2015), 585-600.
%\bibitem[AS]{aprodu-sernesi-excess} M. Aprodu and  E. Sernesi, {\em{Excess dimension for secant loci in symmetric products of curves}}, Collect. Math. (2016). doi:10.1007/s13348-016-0166-2
%\bibitem[AV]{AV} M. Aprodu and C. Voisin, {\em{Green-Lazarsfeld's conjecture for generic curves of large gonality}}, C.R. Math. Acad. Sci. Paris \textbf{336} (2003), 335-339.
%\bibitem[AC1]{AC1} E. Arbarello, M. Cornalba, {\em{Su una congettura di Petri}}, Commentarii Math.\ Helvetici \textbf{56}, (1981),1--37.
%\bibitem[AC]{arbarello-cornalba} E. Arbarello, M. Cornalba, {\em{Footnotes to a paper of Beniamino Segre}}, Mathematische Annalen \textbf{256} (1981), 341-362.
%\bibitem[ACGH]{ACGH1} E. Arbarello, M. Cornalba, P.A. Griffiths and J. Harris, {\em{Geometry of algebraic curves, Volume I}}, Grundlehren der Mathematischen Wissenschaften \textbf{267}, Springer, Heidelberg (1985).
%\bibitem[ACG]{ACG2} E. Arbarello, M. Cornalba, P.A. Griffiths, {\em{Geometry of algebraic curves, Volume II}}, Grundlehren der Mathematischen Wissenschaften \textbf{268}, Springer, Heidelberg (2011).
%\bibitem[AK]{alper-kresch} J. Alper and A. Kresch, {\em{Equivariant versal deformations of semistable curves}}, to appear in Michigan Math.\ J.
%\bibitem[A]{Aprodu} M. Aprodu, {\em{Green-Lazarsfeld gonality conjecture for a generic curve of odd genus}}, International Mathematics Research Notices \textbf{63} (2004), 3409-3416.
\bibitem[AF]{aprodu-farkas} M. Aprodu and G. Farkas, {\em{Green's conjecture for curves on arbitrary K3 surfaces}}, Compositio Math.\ \textbf{147} (2011), 839-851.
%\bibitem[AF2]{AF-covers} M. Aprodu and G. Farkas, {\em{Green's Conjecture for general covers}}, Compact moduli spaces and vector bundles, Contemp. math. \textbf{564} (2012), 211-226.
%\bibitem[AFPRW1]{AFPRW1} M. Aprodu, G. Farkas, S. Papadima, C. Raicu and J. Weyman,  {\em{Topological
%invariants of groups and Koszul modules}}, arxiv:1806.01702.
\bibitem[AFPRW]{AFPRW} M. Aprodu, G. Farkas, S. Papadima, C. Raicu and J. Weyman, {\em{Koszul modules and Green's conjecture}},  Inventiones Math.\ ({\em{to appear}}).
\bibitem[AS]{arbarello-sernesi-petri} E. Arbarello and E. Sernesi, {\em{Petri's approach to the study of the ideal associated to a special divisor}}, Inventiones Math. \textbf{49} (1978), 99-119.


%\bibitem[AF]{clay} M. Aprodu and G. Farkas, {\em{Koszul cohomology and applications to moduli}}, Clay Mathematics Proceedings\ \textbf{14}, American Mathematical Society, Providence, RI (2011).

%\bibitem[AV]{AV} M. Aprodu and C. Voisin, {\em{Green-Lazarsfeld's conjecture for generic curves of large gonality}}, C.R. Math. Acad. Sci. Paris \textbf{336} (2003), 335-339.
%\bibitem[Ba]{babbage} D. W. Babbage, {\em{A note on the quadrics through a canonical curve}}, J. London Math. Soc. \textbf{1} (1939), 310-315.
%\bibitem[Be]{beauville-schottky} A. Beauville, {\em{Prym varieties and the Schottky problem}}, Invent. Math \textbf{41} (1977), 149--196.
%\bibitem[vB1]{bothmer-thesis} H-C Graf v Bothmer, {\em{Geometrische syzygien von kanonischen Kurven}}, Thesis, Universit\"at Bayreuth, 2000.
%\bibitem[vB2]{bothmer-preprint} H-C Graf v Bothmer, {\em{Geometric syzygies of canonical curves of even genus lying on a K3 surfaces}}, arXiv:math/0108078.

%\bibitem[vB3]{bothmer-JPAA} H-C Graf v Bothmer, {\em{Generic syzygy schemes}}, J.\ Pure and Applied Algebra \textbf{208} (2007), 867-876.
%\bibitem[vB4]{bothmer-Transactions} H-C Graf v Bothmer, {\em{Scrollar syzygies of general canonical curves with genus $\le8$}}, Transactions AMS \textbf{359} (2007), 465-488.
%\bibitem[B]{bourbaki-algebra} N. Bourbaki,  {\em{Algebra I: chapters 1-3}}, Volume 1, Springer, 1998.

%\bibitem[B]{beauville-stable} A. Beauville, {\em{Some stable vector bundles with reducible theta divisor}}, Manuscripta Math.\ \textbf{110} (2003), 343-349.
%\bibitem[B]{bopp} C. Bopp, {\em{Syzygies of $5$-gonal canonical curves}}, Documenta Mathematica \textbf{20} (2015), 1055-1069.
%\bibitem[B]{bujokas} G. Bujokas, {\em{The Hurwitz Space of Covers of an Elliptic Curve $E$ and the Severi Variety of Curves in $E \times \PP^1$}}, arXiv:1409.0927.
%\bibitem[BCGGM]{BCGGM} M. Bainbridge, D. Chen, Q. Gendron, S. Grushevsky and M. M\"oller, {\em{Compactification of Strata of Abelian Differentials}}, Duke Math.\ J.\, to appear.
%\bibitem[BE]{BE} D. Buchsbaum and D. Eisenbud, {\em{Generic free resolutions and a family of generically perfect ideals}}, Adv. Math \textbf{18} (1975), 245-301.
%\bibitem[BHT]{bogomolov-hassett-tschinkel} F. Bogomolov, B. Hassett and Y. Tschinkel, {\em{Constructing Rational Curves on K3 Surfaces}}, Duke Math.\ J.\ \textbf{157} (2011), 535-550.
%\bibitem[BS]{bopp-schreyer} C. Bopp and F.-O.Schreyer. {\em{A version of Green's conjecture in positive characteristic}}, arXiv:1803.10481.
%\bibitem[C1]{copp2} M. Coppens, {\em{On G.\ Martens' characterization of smooth plane curves}}, Bull.\ London Math.\ Soc.\ \textbf{20} (1988), 217-220.
%\bibitem[C2]{coppens} M. Coppens, {\em{The existence of k-gonal curves possessing exactly two linear systems $g^1_k$}}, Math. Annalen \textbf{307} (1997), 291-297.

%\bibitem[CCK]{choi-kang-kwak} Y. Choi, P-L Kang and S. Kwak, {\em{Higher linear syzygies of inner projections}}, J. Algebra \textbf{305} (2006), 859-876.

%\bibitem[CE]{casnati-ekedahl} G. Casnati and T. Ekedahl, {\em{Covers of algebraic varieties I. A general structure theorem,
%covers of degree 3, 4 and Enriques surfaces}}, J. Algebraic Geometry \textbf{5} (1996), 439-460.
%\bibitem[CEFS]{CEFS} A. Chiodo, D. Eisenbud, G. Farkas and F.-O. Schreyer, {\em{Syzygies of torsion bundles and the geometry of the level $\ell$ modular variety over $\overline{\mathcal{M}}_g$}}, Invent.\ Math. \textbf{194} (2013), 73-118.
%\bibitem[CP]{ciliberto-pareschi} C. Ciliberto and G. Pareschi, {\em{Pencils of minimal degree on curves on a K3 surface}}, J. Reine Ang. Math. \textbf{460} (1995), 15--36.

%\bibitem[CFVV]{CFVV} E. Colombo, G. Farkas, A. Verra and C. Voisin, {\em{Syzygies of Prym and paracanonical curves of genus 8}}, arXiv:1612.01026.

%\bibitem[C]{coppens} M. Coppens, {\em{One dimensional linear systems of type II on smooth curves}}, Ph.D. Thesis, Utrecht (1983).
%\bibitem[D]{deligne} P. Deligne, {\em{Le Lemme de Gabber}}, Ast\'erisque \textbf{127} (1985), 131-150.
%\bibitem[D]{debarre} O. Debarre, {\em{Sur le probleme de Torelli pour les vari\'et\'es de Prym}} American J.\ Math.\ \textbf{111} (1989), 111--134.
%\bibitem[DP]{deopurkar-patel} A. Deopurkar and A. Patel, {\em{The Picard rank conjecture for the Hurwitz spaces of degree up to five}}, Algebra and Number Theory \textbf{9} (2015), 459-492.
%\bibitem[DM]{donagi-morrison} R. Donagi, D. Morrison, {\em{Linear systems on K3 sections}}, Journal of Differential Geometry \textbf{29} (1989), 49-64.

%\bibitem[Eh]{ehbauer} S. Ehbauer, {\em{Syzygies of points in projective space and applications}}, Zero-dimensional schemes. Proceedings of the international conference held in Ravelo, Italy. 1992.
%\bibitem[E]{eisenbud-orientation} D. Eisenbud, {\em{Green's conjecture: an orientation for algebraists}}, Free resolutions in commutative algebra and algebraic geometry (Sundance, UT, 1990) (1992), 51-78.

%\bibitem[E]{eisenbud-book} D. Eisenbud, {\em{Commutative Algebra: with a view toward algebraic geometry}}, Springer, 2013.
%\bibitem[Ei1]{eisenbud-determinantal} D. Eisenbud, {\em{Linear sections of determinantal varieties}}, American J.\ Math.\ \textbf{110} (1988), 541-575.

%\bibitem[Ei2]{eisenbud-syzygies} D. Eisenbud, {\em{The geometry of syzygies}}, Graduate Texts in Mathematics \textbf{229}, Springer-Verlag, New York, 2005.
%\bibitem[EG]{EG} D. Eisenbud and S. Goto, {\em{Linear free resolutions and minimal multiplicity}}, J.\ Algebra \textbf{88} (1984), 89-133.
%\bibitem[EN]{EN} J. A. Eagon and D. G. Northcott, {\em{Ideals defined by matrices and a certain complex associated with them}}, Proc.\ R.\ Soc.\ Lond.\ A \textbf{269} (1962), 188-204.
%\bibitem[EH]{eisenbud-harris-minimal} D. Eisenbud and J. Harris, {\em{On varieties of minimal degree (A centennial account)}}, Proc. Sympos. Pure Math. \textbf{46} (1987), 3--13.

%\bibitem[EP]{eisenbud-popescu} D. Eisenbud and S. Popescu, {\em{Syzygy ideals for determinantal ideals and the syzygetic Castelnuovo lemma}}, Commutative Algebra, Algebraic Geometry, and Computational Methods \textbf{1996}, 247-258.

%\bibitem[EH]{eisenbud-harris-limit} D. Eisenbud and J. Harris, {\em{Limit linear series: basic theory}}, Inventiones Math.\ \textbf{85} (1986), 337-371.
%\bibitem[EL]{ein-lazarsfeld} L. Ein and R. Lazarsfeld, {\em{The gonality conjecture on syzygies of algebraic curves of
              %large degree}}, Publ. Math. Inst. Hautes \'Etudes Sci. \textbf{122} (2015), 301-313.
 \bibitem[EL]{ein-lazarsfeld-asymptotic} L. Ein and R. Lazarsfeld, {\em{Asymptotic Syzygies of Algebraic Varieties}}, Inventiones Math.\ \textbf{190} (2012), 603-646.
 %\bibitem[EG]{ed-gr} D. Edidin and W. Graham, {\em{Equivariant intersection theory}}, Inventiones Math. \textbf{131} (1998), 595-634.
%\bibitem[F1]{farkas-syzygies} G. Farkas, {\em{Syzygies of curves and the effective cone of $\mm _g$}}, Duke Math.\ J.\ \textbf{135} (2006), 53-98.
%\bibitem[F2]{farkas-koszul} G. Farkas, {\em{Koszul divisors on moduli spaces of curves}}, American J.\ Math.\ \textbf{31} (2009), 819-867.
\bibitem[F]{farkas-progress} G. Farkas, {\em{Progress on syzygies of algebraic curves}}, Lecture Notes of the Unione Matematica Italiana \textbf{21} (Moduli of Curves, Guanajuato 2016), 107-138.
%\bibitem[FHL]{FHL} W. Fulton, J. Harris and R. Lazarsfeld, {\em{Excess linear series on an algebraic curve}}, Proc. American Math. Soc. \textbf{92} (1984), 320-322.

%\bibitem[FK1]{generic-secant} G. Farkas and M. Kemeny, {\em{The generic Green--Lazarsfeld Secant Conjecture}}, Inventiones Math. \textbf{203} (2016), 265-301.
%\bibitem[FK2]{high-level} G. Farkas and M. Kemeny, {\em{The Prym-Green Conjecture for torsion line bundles of high order}}, Duke Math. J.  \textbf{166} (2017) , 1103-1124.
%\bibitem[FK1]{lin-syz}  G. Farkas and M. Kemeny, {\em{Linear syzygies for curves of prescribed gonality}}. Available at \url{http://web.stanford.edu/~mkemeny/preprints/linear-syz.pdf}.
%\bibitem[FK3]{res-odd} G. Farkas and M. Kemeny, {\em{The Resolution of Paracanonical Curves of Odd Genus}}, Geometry and Topology, \emph{to appear}.
%\bibitem[FL]{FaLu} G. Farkas and K. Ludwig, {\em{The Kodaira dimension of the moduli space of Prym varieties}}, J.\ Euro.\ Math. Society \textbf{12} (2010), 755-795.
%\bibitem[FuLa]{fulton-laz-connectedness} W. Fulton and R. Lazarsfeld, {\em{On the connectedness of degeneracy loci and special divisors}}, Acta Math.\ \textbf{146} (1981), 271-283.

%\bibitem[FMP]{farkas-mustata-popa} G. Farkas, M. Musta\c{t}\u{a} and M. Popa, {\em{Divisors on $\cM_{g, g+1}$ and the Minimal Resolution Conjecture for points on
%canonical curves}}, Annales Sci. de L'\'Ecole  Normale Sup\'erieure \textbf{36} (2003), 553-581.
%\bibitem[FP]{faber-pandharipande} C. Faber and R. Pandharipande, {\em{Relative maps and tautological classes}}, Journal of the European Mathematical Society \textbf{7} (2005), 13-49.
 
 %\bibitem[FaP]{faber-pand} C. Faber and R. Pandharipande, {\em{Relative maps and tautological classes.}} J.\ Euro.\ Math.\ Soc.\ \textbf{7} (2005),13-49.
%\bibitem[FP]{far-pand} G. Farkas and R. Pandharipande, {\em{The moduli space of twisted canonical divisors}}, to appear in J.\ Inst.\ Math.\ Jussieu.

%\bibitem[FR]{farkas-rimanyi} G. Farkas and R. Rim\'anyi, {\em{Quadric rank loci on moduli of curves and K3 surfaces}}, arXiv:1707.00756
%\bibitem[Fu]{fulton} W. Fulton, {\em{Intersection theory}}, Vol. 2, Springer, 2013.
%\bibitem[FT]{farkas-tarasca} G. Farkas and N. Tarasca, {\em{Du Val curves and the pointed Brill-Noether theorem}}, arXiv:1606.02725.

%\bibitem[G1]{green-canonical} M. Green, {\em{The canonical ring of a variety of general type}}, Duke Math. J. \textbf{49} (1982), 1087-1113.


\bibitem[G1]{green-koszul} M. Green, {\em{Koszul cohomology and the cohomology of projective varieties}}, J.\  Differential Geo.\ \textbf{19} (1984), 125-171.

%\bibitem[G3]{green-koszul-II} M. Green, {\em{Koszul cohomology and the cohomology of projective varieties. II. }}, Journal of Differential Geometry \textbf{20} (1984), 279-289.
 \bibitem[G2]{green-quadrics} M. Green, {\em{Quadrics of rank four in the ideal of a canonical curve}}, Inventiones Math.\ \textbf{75} (1984), 85-104.
% \bibitem[G3]{green-NL} M. Green, {\em{A new proof of the explicit Noether--Lefschetz theorem.}} J.\ Diff.\ Geo \textbf{27} (1988), 155-159.
%\bibitem[GL1]{green-lazarsfeld-projective} M. Green and R. Lazarsfeld, {\em{On the projective normality of complete linear series on an algebraic curve}}, Inventiones Math. \textbf{83} (1986), 73-90.
\bibitem[Gr]{EGA} A. Grothendieck, {\em{\'El\'ements de g\'eom\'etrie alg\'ebrique : IV. \'Etude locale des sch\'emas et des morphismes de sch\'emas, Seconde partie}}. Publications Math\'ematiques de l'IH\'ES \textbf{24} (1965), 5-231.
% \bibitem[GHS]{GHS} T. Graber, J. Harris and J. Starr, {\em{Families of Rationally Connected Varieties}}, J.\ American Math.\ Soc.\ \textbf{16} (2003), 57-67.
% \bibitem[GL]{green-lazarsfeld-special} M. Green and R. Lazarsfeld, {\em{Special divisors on curves on a K3 surface}}, Inventiones Math. \textbf{89} (1987), 357-370.
 
% \bibitem[GK]{gabai-kazez} D. Gabai and W. H. Kazez, {\em{The classification of maps of surfaces}}, Inventiones Math.\ \textbf{90} (1987), 219-242.

\bibitem[H]{hartshorne} R. Hartshorne, {\em{Algebraic Geometry}}, Graduate Texts in Mathematics \textbf{52}, Springer-Verlag, New York, 1977.
%\bibitem[Ha]{hartshorne-generalized} R. Hartshorne, {\em{Generalized divisors on Gorenstein curves and a theorem of Noether}}, J. Math. Kyoto Univ \textbf{26} (1986), 375-386.

%\bibitem[H]{huybrechts-k3} D. Huybrechts, {\em{Lectures on K3 surfaces.}}, Cambridge University Press \textbf{158}, 2016.
%\bibitem[HL]{huybrechts-lehn} D. Huybrechts and M. Lehn, {\em{The geometry of moduli spaces of sheaves}}, Cambridge University Press, 2010.

%\bibitem[HSV]{HSV} J. Herzog, A. Simis, and W. V. Vasconcelos, {\em{Approximation complexes of blowing-up rings.}} J.\ Algebra \textbf{74} (1982), 466-493.


%\bibitem[HM1]{ha-mo} J. Harris and I. Morrison, {\em{Moduli of Curves}}, Graduate Texts in Mathematics \textbf{187}, Springer-Verlag, New York, 1998.
%\bibitem[HM]{ha-mu} J. Harris and D. Mumford, {\em{On the Kodaira dimension of the moduli space of curves}}, Inventiones Math \textbf{67} (1982), 23-86.
%\bibitem[HR]{hirsch} A. Hirschowitz and S. Ramanan, {\em{New evidence for Green's Conjecture on syzygies of canonical curves}}, Annales Scientifiques de l'\'Ecole Normale Sup\'erieure \textbf{31} (1998), 145-152.

%\bibitem[JPW]{JPW} T. J\'ozefiak, P. Pragacz, J. Weyman, {\em{Resolutions of determinantal varieties and tensor complexes associated with symmetric and antisymmetric matrices}}, Ast\'erique \textbf{87} (1981), 109-189.

%\bibitem[J]{jongmans} F. Jongmans, {\em{Le probl\'eme des s\'eries speciales d'une courbe alg\'ebrique}}, Bull.\ Acad.\ R.\ Sci.\ Belgique \textbf{35} (1949), 1027-1041.
%\bibitem[K1]{kemeny-thesis} M. Kemeny, {\em{Stable Maps and Singular Curves on K3 surfaces}}, Thesis, Universit\"at Bonn (2015).
%\bibitem[K]{kovacs-rational} S. Kov\'acs, {\em{Rational Singularities}}, arXiv:1703.02269.
%\bibitem[KS]{koh-stillman} J. Koh, and M. Stillman, {\em{Linear syzygies and line bundles on an algebraic curve}}, J.\ Algebra \textbf{125} (1989), 120-132.


%\bibitem[K1]{betti-multiple} M. Kemeny, {\em{Betti Numbers of Curves and Multiple-Point Loci}}, arXiv:1804.09221.
%\bibitem[K2]{projecting} M. Kemeny, {\em{Projecting Syzygies of Curves}}, arXiv:1811.01105.
%\bibitem[L]{lascoux} A. Lascous, {\em{Syzygies des vari/'eti/'es d/'eterminales}}, Adv. Math.\ \textbf{30}(1978), 202-237. 

%\bibitem[K2]{kemeny-singular} M. Kemeny, {\em{The Moduli of Singular Curves on K3 Surfaces}}, J. Math. Pures. Appl. \textbf{104} (2015), 882-920.
%\bibitem[Ke]{Keem} C. Keem, {\em{On the variety of special linear systems on an algebraic curve}} Math.\ Annalen \textbf{288} (1990), 309--322.

%\bibitem[Kem]{kempf} G. Kempf, {\em{On the geometry of a theorem of Riemann}}, Annals of Mathematics (1973),178-185.

%\bibitem[Kl]{kleiman} S. Kleiman, {\em{Multiple-point formulae I: Iteration}}, Acta Mathematica \textbf{147} (1981), 13-49.
\bibitem[Ke]{geometric-syzygy} M. Kemeny, {\em{The Geometric Syzygy Conjecture in even genus}}, arXiv:1907.07553.

\bibitem[L]{lazarsfeld-BNP} R. Lazarsfeld, {\em{Brill-Noether-Petri without degenerations}}, J. Differential Geo.\ \textbf{23} (1986), 299-307.

%\bibitem[Ko]{kovacs} S. Kov\'acs, {\em{Rational Singularities}}, arXiv:1703.02269.


%\bibitem[KKZ]{KKZ} A. Kokotov, D. Korotkin and P. Zograf, {\em{Isomonodromic tau function on the space of admissible covers}}, Advances in Mathematics \textbf{227} (2011), 586-600.
%\bibitem[KM]{keel-mckernan} S. Keel and J. M$^c$Kernan, {\em{Contractible extremal rays on $\overline{\mathcal{M}}_{0,n}$}}, arXiv:alg-geom/9607009.
%\bibitem[L]{li-relative} J. Li, {\em{Stable morphisms to singular schemes and relative stable morphisms.}} J.\ Diff.\ Geo.\ \textbf{57} (2001), 509--578.
%\bibitem[LT]{li-tian} J. Li and G. Tian, {\em{Virtual moduli cycles and Gromov-Witten invariants}} J.\ American Math.\ Soc.\ \textbf{11} (1998), 119-174.


%\bibitem[G-Mar]{martens} G. Martens, {\em{On curves on K 3 surfaces}}, Algebraic Curves and Projective Geometry, Springer, Berlin (1989), 174--182. 
%\bibitem[H-Mar]{h-martens} H. Martens, {\em{On the varieties of special divisors on a curve}} J. Reine Ang. Math. \textbf{227} (1967), 111--120.
%\bibitem[May]{mayer} A. Mayer, {\em{Families of K3 surfaces}} Nagoya Math J. \textbf{48} (1972), 1-17.
%\bibitem[Mo]{morrison-large} D. Morrison, {\em{On K3 surfaces with large Picard number}} Invent. math. \textbf{75} (1984), 105-121.

%\bibitem[Mu]{mumford-prym} D. Mumford, {\em{Prym varieties I.}} Contributions to analysis. \textbf{1974}, 325-350.
%\bibitem[N]{noether} M. Noether {\em{Ueber die invariante Darstellung algebraischer Functionen}}, Math.\ Annalen \textbf{17} (1880), 263-284.
%\bibitem[P]{patel-thesis} A. Patel, {\em{The Geometry of Hurwitz space}}, PhD thesis, Harvard University (2013).
%\bibitem[R]{rimanyi} R. Rim\'anyi, {\em{Multiple-point formulas?a new point of view}}, Pacific J. Math \textbf{202} (2002), 475--490.


%\bibitem[K2]{extremal-gonality} M. Kemeny, {\em{The extremal gonality conjecture for curves of arbitrary gonality}}, arXiv:1512.00212.
%\bibitem[Kn]{knutsen} A. Knutsen, {\em{On $k$-th order embeddings of K3 surfaces and Enriques surfaces}}, Manuscripta Math. \textbf{104} (2001), 211-237.
%\bibitem[M]{h-martens} H. Martens, {\em{On the varieties of special divisors on a curve}}, Journal f\"ur die reine und angewandte Mathematik \textbf{227} (1967), 111-120.
%\bibitem[P]{peeva} I. Peeva, {\em{Graded syzygies}}, Springer, 2010.

%\bibitem[R]{rathmann} J. Rathmann, {\em{An effective bound for the gonality conjecture}}, arXiv:1604.06072.
\bibitem[RS]{raicu-sam} C. Raicu and S. Sam, {\em{Bi-graded Koszul modules, K3 carpets, and Green's conjecture}}, arXiv:1909.09122
%\bibitem[Ry]{rydh} D. Rydh {\em{The canonical embedding of an unramified morphism in an étale morphism}}, Math.\ Zeit.\ \textbf{268} (2011), 707-723.
%\bibitem[S1]{sagraloff} M. Sagraloff, {\em{Special linear series and syzygies of canonical curves of genus $9$}}, Thesis, Saarland University, 2005.
%\bibitem[S2]{sernesi-def} E. Sernesi, {\em{Deformations of algebraic schemes}}, Springer, 2007.
%\bibitem[SP]{stacks} The Stacks Project Authors, {\em{Stacks Project}}, \url{http://stacks.math.columbia.edu}, 2019.
%\bibitem[S]{segre} B. Segre, {\em{Sui moduli delle curve poligonali, e sopra un complemento al teorema di esistenza di Reimann}}, Math. Annalen \textbf{100} (1928), 537-551.
%\bibitem[Sch]{schreyer1} F.-O. Schreyer, {\em{Syzygies of canonical curves and special linear series}}, Math. Ann. \textbf{275} (1986), 105-137.
%\bibitem[Sch2]{sch} F.-O. Schreyer, {\em{Green's conjecture for the general $p$-gonal curve of large genus}}, In: Algebraic curves and projective geometry, Springer Lecture Notes \textbf{1389} (1988), 254-260.
%\bibitem[Sch2]{schreyer-topics} F.-O. Schreyer, {\em{Some topics in computational algebraic geometry}}, In: Advances in algebra and geometry (Hyderabad, 2001), Hindustan Book Agency (2003), 263-278.
%\bibitem[SSW]{SSW} J. Schicho, F.-O. Schreyer and M. Weimann, {\em{Computational aspects of gonal maps and radical parametrization of curves}}, Appl.\ Algebra Engrg.\ Comm.\ Comput.\ \textbf{24} (2013), 313-341.
%\bibitem[Te]{montserrat} M. Teixidor i Bigas {\em{Green's conjecture for the generic $r$-gonal curve of genus $g\geq 3r-7$}} Duke Math.\ J.\ \textbf{111} (2002), 195-222.
%\bibitem[Ta]{tannenbaum} A. Tannenbaum, {\em{Families of curves with nodes on K3 surfaces}}, Math.\ Ann.\ \textbf{260} (1982), 239-253.
%\bibitem[Te]{Te} M. Teixidor, {\em{Syzygies using vector bundles}}, Transactions of the American Mathematical Society \textbf{359} (2007), 897-908.
%\bibitem[Ty]{Ty} I. Tyomkin, {\em{On Severi varieties on Hirzebruch surfaces}},  International Mathematical  Research Notices \textbf{23} (2007).
%\bibitem[vdGK]{vgeer-kouvidakis} G. van der Geer and A. Kouvidakis, { \em{The Hodge Bundle on Hurwitz Spaces}} Pure and Applied Mathematics Quarterly \textbf{7}, 1927-1308, (2011).
%\bibitem[V1]{voisin-NL} C. Voisin, {\em{Une pr\'ecision concernant le th\'eoreme de Noether}}, Math. Annalen \textbf{280} (1988), 605-611.
%\bibitem[V1]{voisin-deformation} C. Voisin, {\em{Deformation des Syzygies et Theorie de Brill?Noether}}, Proc.\ London Math.\ Soc.\ \textbf{3.3} (1993), 493-515.
\bibitem[V1]{V1} C. Voisin, {\em{Green's generic syzygy conjecture for curves of even genus lying on a $K3$ surface}}, J.\ European Math. Society \textbf{4} (2002), 363-404.
\bibitem[V2]{V2} C. Voisin, {\em{Green's canonical syzygy conjecture for generic curves of odd genus}}, Compositio Math.\ \textbf{141} (2005), 1163-1190.
%\bibitem[Wei]{weibel} C. Weibel, {\em{An introduction to homological algebra}}, Cambridge university press, 1995.
\bibitem[W]{weyman-sym-ext} J. Weyman, {\em{Resolutions of the exterior and symmetric powers of a module}}, J.\ Algebra \textbf{58} (1979), 333-341.
%\bibitem[W2]{weyman-book} J. Weyman, {\em{Cohomology of vector bundles and syzygies}}, Cambridge University Press, 2003.


\end{thebibliography}
\end{document}